\documentclass{amsart}

\usepackage[utf8]{inputenc}

\usepackage{geometry}
\usepackage{amsthm}
\usepackage{amssymb}
\usepackage{amsmath}
\usepackage{comment}
\usepackage{thm-restate}
\usepackage{url} 
\usepackage{hyperref}
\usepackage[noabbrev,capitalise]{cleveref}
\usepackage{dsfont}

\usepackage{multicol}

\usepackage{enumitem}

\usepackage{color}
\usepackage[normalem]{ulem}

\usepackage{tikz}

\newcommand{\mc}[1]{\mathcal{#1}}

\renewcommand{\v}{\textup{\textsf{v}}}

\newcommand{\dist}{\textup{\textsf{dist}}}

\newcommand{\dash}{\text{-}}

\usepackage{caption}
\usepackage[labelformat=simple]{subcaption}

\theoremstyle{plain}
\newtheorem{thm}{Theorem}[section]
\newtheorem{lem}[thm]{Lemma}

\newtheorem{conj}[thm]{Conjecture}

\newtheorem{claim}{Claim}
\counterwithin*{claim}{section}

\theoremstyle{definition}

\newcommand{\N}{\mathbb{N}}

\usepackage{algorithm}
\usepackage{algpseudocode}

\allowdisplaybreaks

\makeatletter
\@namedef{subjclassname@2020}{%
\textup{2020} Mathematics Subject Classification}
\makeatother

\begin{document}

\title{On an induced version of Menger's theorem}

\thanks{The first author is supported by the Institute for Basic Science (IBS-R029-C1). The second and fourth authors are supported by the Natural Sciences and Engineering Research Council of Canada (NSERC). Les deuxi\`eme et quatri\`eme auteurs sont supportés par le Conseil de recherches en sciences naturelles et en génie du Canada (CRSNG). The third author is funded by an ETH Z\"{u}rich Postdoctoral Fellowship.}


\author{Kevin Hendrey}
\address{Discrete Mathematics Group, Institute for Basic Science (IBS), Daejeon, South Korea}
\email{kevinhendrey@ibs.re.kr}
\urladdr{sites.google.com/view/kevinhendrey/}

\author{Sergey Norin}
\address{Department of Mathematics and Statistics, McGill University, Montr\'eal, Canada}
\email{snorin@math.mcgill.ca}
\urladdr{www.math.mcgill.ca/snorin/}

\author{Raphael Steiner}
\address{Institute of Theoretical Computer Science, Department of Computer Science, ETH Z\"{u}rich, Switzerland}
\email{raphaelmario.steiner@inf.ethz.ch}
\urladdr{sites.google.com/view/raphael-mario-steiner/}

\author{J\'er\'emie Turcotte}
\address{Department of Mathematics and Statistics, McGill University, Montr\'eal, Canada}
\email{mail@jeremieturcotte.com}
\urladdr{www.jeremieturcotte.com}

\begin{abstract}
	We prove Menger-type results in which the obtained paths are pairwise non-adjacent, both for graphs of bounded maximum degree and, more generally, for graphs excluding a topological minor. We further show better bounds in the subcubic case, and in particular obtain a tight result for two paths using a computer-assisted proof.
\end{abstract}

\maketitle

\section{Introduction}

Given a graph $G$ and $X,Y\subseteq V(G)$, we say a set of vertices $Z$ \emph{separates} $X$ and $Y$ if $Z$ intersects every $X\dash Y$-path\footnote{An $X$-$Y$-path is defined as a path that has one endpoint in $X$ and the opposite endoint in $Y$. Notably, this definition also allows a single vertex in $X\cap Y$ to qualify as an $X$-$Y$-path (of length $0$).}. In general, we say two paths are disjoint if they do not share any vertices. Menger's theorem is a fundamental result of graph theory, relating the existence of many disjoint paths between two sets of vertices in a graph with the absence of small separators. 

\begin{thm}[Menger's theorem \cite{menger_zur_1927}]\label{thm:menger}
    If $k\in \N$, $G$ is a graph and $X,Y\subseteq V(G)$, then there exists either
    \begin{enumerate}[label=\normalfont(\arabic*)]
        \item $k$ pairwise disjoint $X\dash Y$-paths, or
        \item a set of less than $k$ vertices which separates $X$ and $Y$.
    \end{enumerate}
\end{thm}

It is a natural question to ask under which circumstances we can guarantee the paths in point (1) to be far apart from each other in the graph metric, rather than just disjoint. Georgakopoulos and Papasoglu \cite{georgakopoulos_graph_2023}, motivated by questions in metric geometry, and Albrechtsen et al.\ \cite{albrechtsen_induced_2023}, have conjectured the following ``Coarse Menger's theorem''.

\begin{conj}\label{conj:coarse}
    For every $k\in \N$, there exists $c=c(k)\in \N$ satisfying the following. If $d\in \N$, $G$ is a graph and $X,Y\subseteq V(G)$, then there exists either
    \begin{enumerate}[label=\normalfont(\arabic*)]
        \item $k$ disjoint $X\dash Y$-paths $P_1,\dots,P_k$ such that $\dist_G(P_i,P_j) \geq d$ for all distinct $i,j$, or
        \item a set $Z \subseteq V(G)$ of size less than $k$ such that $B_G(Z,cd)$ separates $X$ and $Y$.
    \end{enumerate}
\end{conj}

We note that the conjecture of Albrechtsen et al.\ is in fact stronger, as it does not allow $c$ to depend on $k$. McCarty and Seymour have shown that it is sufficient to prove the conjecture for the case $d=3$ for the entire conjecture to hold, see \cite[Theorem 4]{albrechtsen_induced_2023}.

Both Georgakopoulos and Papasoglu \cite{georgakopoulos_graph_2023} and Albrechtsen et al.\ \cite{albrechtsen_induced_2023} have shown that \cref{conj:coarse} holds for $k=2$, the constant of $129$ below is from the latter authors.

\begin{thm}[\cite{georgakopoulos_graph_2023,albrechtsen_induced_2023}]\label{thm:2coarse}
   If $d\in \N$, $G$ is a graph and $X,Y\subseteq V(G)$, then there exists either
   \begin{enumerate}[label=\normalfont(\arabic*)]
       \item two disjoint $X\dash Y$-paths $P_1, P_2$ such that $\dist_G(P_1,P_2)\geq d$, or
       \item $z \in V(G)$ such that $B_G(z,129d)$ separates $X$ and $Y$. 
   \end{enumerate}
\end{thm}

A natural  special case of \cref{conj:coarse} is the case when the considered graph $G$ has maximum degree bounded by a constant $\Delta$. Indeed, in this case one can upper-bound the size of the ball $B_G(Z,cd)$ in the statement of \cref{conj:coarse} by $|B_G(Z,cd)|\le |Z|\cdot \sum_{i=0}^{cd}{\Delta^i}<\Delta^{cd+1}k$. In particular, if the strong version of \cref{conj:coarse} proposed by Albrechtsen et al. holds (that is, with $c$ independent of $k$), then also the following must be true.

\begin{conj}\label{conj:maxdegree}
For every $d, \Delta \in \mathbb{N}$ there exists a constant $C=C(d,\Delta)>0$ such that the following holds. If $k \in \mathbb{N}$, $G$ is a graph with $\Delta(G)\le \Delta$ and $X, Y\subseteq V(G)$, then there exists either
 \begin{enumerate}[label=\normalfont(\arabic*)]
       \item $k$ $X\dash Y$-paths pairwise at distance at least $d$ in $G$, or
       \item a set of less than $Ck$ vertices in $G$ which separates $X$ and $Y$.
   \end{enumerate}
\end{conj}

In the first main result of this paper, namely \cref{thm:strongcolouring} below, we prove \cref{conj:maxdegree} in the first non-trivial case when $d=2$, that is, when we look for a family of disjoint $X\dash Y$-paths that are pairwise non-adjacent\footnote{Since for every $X$-$Y$-path $P$ in a graph $G$ there exists an induced $X$-$Y$-path $P'$ such that $V(P')\subseteq V(P)$, one can see that the existence of a family of $k$ pairwise non-adjacent $X$-$Y$-paths is equivalent to the existence of a family $\mathcal{P}$ of $k$ different $X$-$Y$-paths, such that the union of the paths in $\mathcal{P}$ forms an induced subgraph of $G$. This explains the naming of the paper.}. For brevity, when saying that paths are pairwise non-adjacent, they will also be meant to be disjoint.

To state \cref{thm:strongcolouring} concisely we need a bit of terminology. Let $G$ be a graph. We say $M\subseteq E(G)$ is an induced matching if $\dist_G(e_1,e_2)\geq 2$ for every distinct $e_1,e_2\in M$. A \emph{strong edge colouring} of $G$ is a partition of the edges of $G$ into induced matchings. In other words, the edges are coloured such that no two edges of the same colour are adjacent. The \emph{strong chromatic index} of $G$, denoted by $\chi_s'(G)$, is the smallest number of matchings in a strong edge colouring of $G$. The strong chromatic index is well studied, and there are many known bounds depending on the maximum degree $\Delta(G)$ of $G$, as well as for more specific classes of graphs. In general, $\chi'_s(G)\leq 2\Delta(G)(\Delta(G)-1)+1$, which can be seen by counting the number of edges at distance at most 2 of any edge and colouring greedily. For large enough $\Delta$, currently the best bound, by Hurley, de Verclos and Kang \cite{hurley_improved_2022}, is $\chi_s'(G)\leq 1.772\Delta^2$ when $\Delta(G)\leq \Delta$.

\begin{thm}\label{thm:strongcolouring}
    If $k\in \N$, $G$ is a graph and $X,Y\subseteq V(G)$, then there exists either
    \begin{enumerate}[label=\normalfont(\arabic*)]
        \item $k$ pairwise non-adjacent $X\dash Y$-paths, or
        \item a set of less than $2^{\chi_s'(G)}k$ vertices which separates $X$ and $Y$.
    \end{enumerate}
\end{thm}

Given the bounds on $\chi_s'(G)$ mentioned above, we may obtain the $d=2$ case of \cref{conj:maxdegree} as a direct corollary of \cref{thm:strongcolouring}. 

By Menger's theorem (\cref{thm:menger}), \cref{thm:strongcolouring} is equivalent to the following result.

\begin{thm}\label{thm:strongcolouring2}
    If $k\in \N$, $G$ is a graph, $X,Y\subseteq V(G)$ and there exists at least $2^{\chi_s'(G)}k$ pairwise disjoint $X\dash Y$-paths, then there exist $k$ pairwise non-adjacent $X\dash Y$-paths in $G$.
\end{thm}

The idea behind the proof is as follows. Given a large number of disjoint $X\dash Y$-paths and a strong colouring of the edges in between the paths, we contract all edges of a colour class (say, green) and apply Menger's theorem to find many disjoint $X\dash Y$-paths in the contracted graph, which we can then lift back to the original graph. By this contraction, we will be guaranteed to not have any green edges between the paths, and the strong colouring will guarantee that there are no edges of the original paths that go between the new paths. After repeating this argument for every colour, we find a collection of pairwise non-adjacent paths.

We will prove \cref{thm:strongcolouring2} in \cref{sec:maxdegree}. In fact, given that we do not need to colour the edges of the original paths, we can obtain an improvement over the constant $2^{\chi_s'(G)}$, which is most significant when the maximum degree is small. In particular, in \cref{sec:subcubic}, we will show the following two results. For brevity, if $\mathcal P$ is a collection of subgraphs (typically, of paths) of a graph $G$, then write $V(\mathcal P)=\bigcup_{P\in \mathcal P} V(P)$ and $E(\mathcal P)=\bigcup_{P\in \mathcal P} E(P)$.

\begin{restatable}{thm}{kpathssubcubic}\label{thm:kpathssubcubic}
    If $G$ is a graph, $X,Y\subseteq V(G)$, and there exists a collection $\mathcal P$ of at least $16k$ disjoint $X\dash Y$-paths such that every vertex in $V(\mathcal P)$ is incident to at most one edge in $E(G)\setminus E(\mathcal P)$, then there exist at least $k$ pairwise non-adjacent $X\dash Y$-paths in $G$.
\end{restatable}

This is a large improvement over the constant $2^{10}$ which would be obtained from \cref{thm:strongcolouring2} with the bound of $\chi_s'(G)\leq 10$ when $\Delta(G)\leq 3$ proved independently by Andersen \cite{andersen_strong_1992} and by Horák, Qing and Trotter \cite{horak_induced_1993}.

For the case $k=2$, we will show in a computer-assisted proof the following tight result.

\begin{restatable}{thm}{twopathssubcubic}\label{thm:twopathssubcubic}
    If $G$ is a graph, $X,Y\subseteq V(G)$, and there exists a collection $\mathcal P$ of five disjoint $X\dash Y$-paths such that every vertex in $V(\mathcal P)$ is incident to at most one edge in $E(G)\setminus E(\mathcal P)$, then there exist  two pairwise non-adjacent $X\dash Y$-paths in $G$.

    Furthermore, the statement does not necessarily hold if either
    \begin{enumerate}[label=\normalfont(\alph*)]
        \item $\mathcal P$ contains four paths instead of five, or
        \item we replace the condition that every vertex in $V(\mathcal P)$ is incident to at most one edge in $E(G)\setminus E(\mathcal P)$ by the condition that the maximum degree of $G$ is three.
    \end{enumerate}
\end{restatable}

Both the proof methods of Theorems \ref{thm:kpathssubcubic} and \ref{thm:twopathssubcubic} and (b) in the latter indicate that the maximum degree of $G-E(\mathcal{P})$ is a more natural parameter to bound than the maximum degree of $G$.

A direct consequence of (1) in this result is that the $16k$ in \cref{thm:kpathssubcubic} cannot be improved below $4k-3$ (consider a disjoint union of $k-1$ copies of a graph containing $4$ disjoint $X\dash Y$-paths but no two non-adjacent $X\dash Y$-paths).

We say a graph $H$ is a \emph{topological minor} of a graph $G$ if $G$ contains a subdivision\footnote{As usual, by a \emph{subdivision} of $H$ we here mean any graph that is isomorphic to a graph that can be obtained from $H$ by replacing its edges by paths of positive length.} of $H$ as a subgraph. In \cref{sec:subdivision}, using the structure theorem for graphs excluding a topological $K_r$-minor first proved by Grohe and Marx \cite{grohe_structure_2015}, as well as Erde and Wei{\ss}auer \cite{erde_short_2019}, we then generalize our induced Menger's theorem to the class of graphs excluding the complete graph $K_r$ as a topological minor.

\begin{thm}\label{thm:topologicalminor}
For every $r\in \N$, there exists $c=c_{\ref{thm:topologicalminor}}(r)>0$ such that the following holds. 

If $G$ is a graph not containing $K_r$ as a topological minor and $X, Y \subseteq V(G)$, then there exists either
\begin{enumerate}[label=\normalfont(\arabic*)]
        \item $k$ pairwise non-adjacent $X\dash Y$-paths, or
        \item a set of less than $ck$ vertices which separates $X$ and $Y$.
    \end{enumerate}
\end{thm}

\section{Graphs with bounded maximum degree}\label{sec:maxdegree}

\cref{thm:strongcolouring2} follows directly from the following result, by taking $\mathcal M$ as the colour classes of edges in $E(G[V(\mathcal P)])\setminus E(\mathcal P)$ (in other words, the edges not in $\mathcal P$ but with both ends in vertices in $\mathcal P$) in a strong edge colouring of $G$.

\begin{thm}\label{thm:strongcolouring3}
    If $m,k\in \N$, $G$ is a graph, $X,Y\subseteq V(G)$, $\mathcal P$ is a collection of $2^mk$ pairwise disjoint $X\dash Y$-paths and $\mathcal M$ is a partition of $E\left(G[V(\mathcal P)]\right)\setminus E(\mathcal P)$ into $m$ induced matchings of $G$, then there exist $k$ pairwise non-adjacent $X\dash Y$-paths in $G$.
\end{thm}

\begin{proof}
    First note that we may without loss of generality assume that $V(G)=V(\mathcal P)$, by restricting $G,X,Y$ to $ V(\mathcal P)$. This implies that $E(G)=E(\mathcal P)\cup (\cup \mathcal M)$, where $\cup \mathcal M:=\bigcup_{M\in \mathcal M} M$.

    We prove the statement by induction on $m$. If $m=0$, then $\mathcal M=\emptyset$. In particular, $E\left(G[V(\mathcal P)]\right)\setminus E(\mathcal P)=\emptyset$, and so the $2^0k=k$ paths of $\mathcal P$ are pairwise non-adjacent.

    We now show the inductive step. We may assume that the paths of $\mathcal P$ are chosen such that the sum of the lengths of paths in $\mathcal{P}$ is smallest possible among all collections of $k$ disjoint $X\dash Y$-paths. This immediately implies that every $P \in \mathcal{P}$ is an induced path in $G$, and that every path $P \in \mathcal{P}$ intersects $X$ and $Y$ only in its endpoints.
    
    Let $M\in \mathcal M$, chosen arbitrarily, and write $\mathcal M^\ast:=\mathcal M\setminus \{M\}$. We define $G'=G-E(\mathcal M^\ast)$, and let $G''$ be obtained from $G'$ by contracting the edges of $M$.\footnote{An edge is  \emph{contracted} by identifying its end vertices, and removing any resulting loops and parallel edges.} Let $f:V(G)\rightarrow V(G'')$ denote the mapping of vertices underlying the resulting contraction. The sets corresponding to $X,Y$ in $G''$ are then $X':=f(X),Y':=f(Y')$.

    By Menger's theorem (\cref{thm:menger}), there exists either
    \begin{enumerate}
        \item a collection $\mathcal P'$ of $2^{m-1}k$ disjoint $X'\dash Y'$-paths in $G''$, or
        \item a set $Z'$ of size less than $2^{m-1}k$ vertices separating $X'$ and $Y'$ in $G''$.
    \end{enumerate}
    
    First suppose we are in case (2). Let $Z:=f^{-1}(Z')$. As we have contracted by a matching, the preimage of any vertex in $G''$ is of size at most two, and so $|Z|\leq 2|Z'|<2^mk$. We claim $Z$ separates $X$ and $Y$ in $G'$, which would be a contradiction as there are at least $2^mk$ pairwise disjoint $X\dash Y$-paths in $G'$ (the paths in $\mathcal P$ are not affected by removing $M^\ast$). If $P$ is an $X\dash Y$-path in $G'$, then $f(V(P))$ corresponds to the vertex set of a  $X'\dash Y'$ walk in $G''$, from which we can extract an $X'\dash Y'$-path $P'$. Given that $Z'$ separates $X'$ and $Y'$ in $G''$, there exists $v'\in Z'\cap V(P')\neq \emptyset$. By construction of $P'$, there exists $v\in V(P)$ such that $f(v)=v'$. By definition of $Z$, $v\in Z$. Hence, $v\in Z\cap V(P)$ as desired. Therefore, we are necessarily in case (1).

    Given a path $P'\in \mathcal P'$, it is easily seen that $G'[f^{-1}(V(P'))]$ is connected and so there exists an $X\dash Y$ path $P$ in $G'$ such that $f(V(P)) \subseteq V(P')$. We may further suppose that this path is induced in $G'$. We call such a path a \emph{lift} of $P'$. Let $\mathcal P_2$ be the collection of lifts of paths of $\mathcal P'$. Given that the paths of $\mathcal P'$ are pairwise disjoint, so are those in $\mathcal P_2$.

    Let $\mathcal M_2$ be the collection of matchings of $\mathcal M^*$ restricted to edges in $E\left(G[V(\mathcal P_2)]\right)\setminus E(\mathcal P_2)$.

    We claim $\mathcal M_2$ partitions the edges of $E\left(G[V(\mathcal P_2)]\right)\setminus E(\mathcal P_2)$. The fact that the matchings in $M_2$ are pairwise disjoint is direct from their construction as restrictions of matchings in $\mathcal M^*$. Let $e=uv\in E\left(G[V(\mathcal P_2)]\right)\setminus E(\mathcal P_2)$, and suppose for a contradiction that $e$ is not in any matching of $\mathcal M_2$.
    
    First suppose that $e\in M$. We cannot have $u,v\in V(P)$ for $P\in \mathcal P_2$, since $e\notin E(\mathcal P_2)$, $M \subseteq E(G')$ and $P$ is induced in $G'$. Hence, $u\in V(P_1)$ and $v\in V(P_2)$ for distinct paths $P_1,P_2\in \mathcal P_2$. However, $P_1,P_2$ are  lifts of paths in $G''$, say $P_1',P_2'\in \mathcal P'$. In particular, $f(u)\in V(P_1')$ and $f(v)\in V(P_2')$. As $uv\in M$, $f(u)=f(v)$, and so $P_1'$ and $P_2'$ are not disjoint, which is a contradiction to the choice of $\mathcal P'$. Hence, we may now suppose that $e\notin M$.
    
    Given that $e\notin M$ and $e$ is not in any matching of $\mathcal M^*$, $e\notin \cup \mathcal M$. By our first assumption, necessarily $e\in E(\mathcal P)$. If we show that both $u$ and $v$ are incident to some edges of $M$, this would be a contradiction to the fact that $M$ is a strong matching.
    
    We will show that $u$ is incident to some edge of $M$; the proof for $v$ is analogous. Let $P\in \mathcal P_2$ be the path such that $u\in V(P)$. There are two cases to consider. First suppose $u$ is not an endpoint of $P$, i.e.\ $u$ has distinct neighbours $z_1,z_2\in V(P)$. It is impossible that both $uz_1,uz_2\in E(\mathcal P)$, given that this is a collection of paths (so $u$ cannot be incident to three edges of $E(\mathcal P)$) and we know that $e\in E(\mathcal P)$; without loss of generality say $uz_1\in \cup \mathcal M$. Recall that $\mathcal P_2$ is a collection of paths in $G'$ and so necessarily $uz_1\in M$, as desired. Now suppose that $u$ is an endpoint of $P$, hence $u\in X\cup Y$. If $u\in X\cap Y$, then $e$ would not appear in $\mathcal P$, as we have assumed those paths to be as short as possible. Hence, $u\notin Y$ and so there exists some edge $uz\in E(\mathcal P_2)$. By our hypothesis that the paths in $\mathcal P$ are shortest possible, $u$ is not an interior vertex of any path in $\mathcal P$, i.e.\ $u$ appears in at most one edge of $\mathcal P$, which we already know to be $e$. Hence $uz\notin E(\mathcal P)$ and so $uz\in \cup \mathcal M$. By the same argument as previously, $uz\in M$. This completes the proof of the claim.

    Note that $|\mathcal P_2|=|\mathcal P'|=2^{m-1}k$ and $|\mathcal M_2|=|\mathcal M^\ast|=|\mathcal M|-1=m-1$. Hence, by the induction hypothesis applied to $G'$, we obtain $k$ pairwise non-adjacent $X\dash Y$-paths in $G$, as desired.
\end{proof}

\section{Excluding a topological minor}\label{sec:subdivision}

In this section, we will prove \cref{thm:topologicalminor}. We first need some definitions and notation. Let $G$ be a graph.
If $S\subseteq V(G)$, we write $G[S]$ for the subgraph of $G$ induced by $S$.

A \emph{separation} in $G$ is to be understood as a pair $(A,B)$ of subsets of $V(G)$ such that $A \cup B=V(G)$ and there exists no edge in $G$ with endpoints in $A\setminus B$ and $B \setminus A$. It is slightly unusual but convenient for us to allow in this definition also degenerate cases in which $A\subseteq B$ or $B\subseteq A$. Given a separation $(A,B)$ of $G$, we call $A\cap B$ its \emph{separator} and refer to $|A\cap B|$ as the \emph{order} of the separation $(A,B)$.

A \emph{tree-decomposition} of $G$ is a pair $(T,\mathcal{V})$, where $T$ is a tree and $\mathcal{V}=(V_t)_{t \in V(T)}$ is a collection of subsets of $V(G)$ satisfying the following properties:
   \begin{itemize}
        \item for every $v \in V(G)$, the set $\{t \in V(T) : v \in V_t\}$ induces a non-empty subtree of $T$, and
        \item for every $uv \in V(G)$, there exists at least one $t \in V(T)$ such that $u, v \in V_t$.
    \end{itemize}
Given a tree decomposition $(T,\mathcal{V})$ of $G$, for every edge $e=t_1t_2\in E(T)$, we denote $S(e):=V_{t_1} \cap V_{t_2}$ and call $\max_{e \in E(T)}{|S(e)|}$ the \emph{adhesion} of the tree-decomposition $(T,\mathcal{V})$. Given a vertex $t \in V(T)$, the \emph{torso at $t$}, denoted by $\tau(t)$, is defined as the graph obtained from $G[V_t]$ by adding, for every edge $e \in E(T)$ incident to $t$, an edge between any two non-adjacent vertices in $S(e)$, in other words we make $S(e)$ a clique for every incident edge $e$ of $t$.

For every edge $e=t_1t_2 \in E(T)$, there exists a natural corresponding separation in $G$, namely
$$\left(\bigcup_{t \in (T-e)(t_1)}{V_t},\bigcup_{t \in (T-e)(t_2)}{V_t}\right),$$ where $(T-e)(t_i)$ denotes the set of vertices of the unique component of $T-e$ that contains $t_i$. From the definition of a tree decomposition it is not hard to see that this indeed is a separation in $G$, with $V_{t_1}\cap V_{t_2}=\left(\bigcup_{t \in (T-e)(t_1)}{V_t}\right) \cap \left(\bigcup_{t \in (T-e)(t_2)}{V_t}\right)$ being the corresponding separator. 

Finally, we say a graph $H$ is a \emph{minor} of $G$ if a graph isomorphic to $H$ can be obtained from $G$ be removing vertices and edges, and contracting edges. It is direct that if $G$ contains $H$ as a topological minor, it also contains $H$ as a minor.

The following structure theorem is a key element of our proof of \cref{thm:topologicalminor}. We use the exact statement of Erde and Weißauer \cite[Theorem 4]{erde_short_2019}, see also Grohe and Marx \cite[Theorem 4.1]{grohe_structure_2015}.

\begin{thm}[\cite{grohe_structure_2015,erde_short_2019}]\label{thm:grohemarx}
If $r\in \N$ and $G$ is a graph excluding $K_r$ as a topological minor, then $G$ admits a tree-decomposition of adhesion less than $r^2$ such that every torso either 
\begin{enumerate}[label=\normalfont(\arabic*)]
    \item has fewer than $r^2$ vertices of degree at least $2r^4$, or
    \item is $K_{h}$-minor-free, for $h=2r^2$. 
\end{enumerate}
\end{thm}

Broadly speaking, our proof of \cref{thm:topologicalminor} will proceed as follows. Given a collection of $X\dash Y$-paths in a smallest counterexample $G$, we will apply \cref{thm:grohemarx} and find a torso of the tree decomposition which intersects every path in the collection. Then, in order to find the desired collection of paths, we will either apply our result for bounded maximum degree (\cref{thm:strongcolouring2}), if we are in case (1), or use the following lemma, if we are in case (2).

\begin{lem}\label{lemma:minorfree}
    If $h,k\in \N$, $G$ is a $K_h$-minor-free graph, $X, Y \subseteq V(G)$ and there exists $k$ pairwise disjoint $X\dash Y$-paths in $G$, then there exists at least $\frac{k}{2(h-1)}$ pairwise non-adjacent $X\dash Y$-paths in $G$.
\end{lem}
\begin{proof}
    Let $\mathcal{P}$ be a collection of $k$ disjoint $X\dash Y$-paths in $G$. Let $H$ be the minor of $G$ obtained from $G\left[\bigcup_{P \in \mathcal{P}}{V(P)}\right]$ by contracting each path $P \in \mathcal{P}$ into a single vertex. In this way, the vertices of $H$ have a natural one-to-one correspondence with the paths in $\mathcal{P}$, and two vertices in $H$ are adjacent if and only if the corresponding paths in $\mathcal{P}$ are adjacent. Since $G$ is $K_h$-minor-free, so is $H$. Hence, by a classical result of Duchet and Meyniel~\cite{duchet_hadwigers_1982}, we have that $H$ contains an independent set of size at least $\alpha(H)\ge \frac{\v(H)}{2(h-1)}=\frac{k}{2(h-1)}$. The subcollection $\mathcal{P}'\subseteq \mathcal{P}$ corresponding to this independent set in $H$ now consists of pairwise non-adjacent $X\dash Y$-paths, as desired. 
\end{proof}

\cref{thm:topologicalminor} follows directly from the following result, by applying Menger's theorem (\cref{thm:menger}) and choosing $c_{\ref{thm:topologicalminor}}(t)\geq \frac{1}{\varepsilon_{\ref{thm:topologicalminor2}}(t)}$. The additive $\frac{1}{2}$ is used solely for formal reasons, as it simplifies the inductive proof.

\begin{thm}\label{thm:topologicalminor2}
For every $r\in \N$, there exists $\varepsilon=\varepsilon(r)>0$ such that the following holds. 

If $G$ is a graph not containing $K_r$ as a topological minor, $X, Y \subseteq V(G)$, $k \in \mathbb{N}$ and there are $k$ pairwise disjoint $X\dash Y$-paths in $G$, then there also exists a family of at least $\varepsilon k+\frac{1}{2}$ pairwise non-adjacent $X\dash Y$-paths in~$G$.
\end{thm}

\begin{proof}
Fix $r\in \N$; we prove the statement with the constant $\varepsilon(r):=2^{-(8r^8+3)}$. Towards a contradiction, suppose the claim is not true, and consider a counterexample $G$ with $\v(G)$ minimum. Hence, there exist $X, Y \subseteq V(G)$ and $k\in \N$ such that on the one hand, there exists a collection $\mathcal{P}$ consisting of $k$ pairwise disjoint $X\dash Y$-paths in $G$, but on the other hand, every collection $\mathcal{Q}$ of pairwise non-adjacent $X\dash Y$-paths in $G$ has size less than $\varepsilon k+\frac{1}{2}$. Note that the latter fact implies that $1<\varepsilon k+\frac{1}{2}$, so $k > \frac{1}{2\varepsilon}$. Our next claim uses the minimality assumption on $G$ to guarantee that for every separation $(A,B)$ in $G$ of sufficiently small order, one of its two sides must intersect all paths in $\mathcal{P}$.

\begin{claim}\label{claim:separation}
    If $(A,B)$ is a separation in $G$ of order $|A\cap B|<2^{8r^8+1}$, then $V(P) \cap A \neq \emptyset$ for every $P \in \mathcal{P}$ or $V(P) \cap B \neq \emptyset$ for every $P \in \mathcal{P}$.
\end{claim}

\begin{proof}[Proof of \cref*{claim:separation}.]
Suppose towards a contradiction that there exist two paths $P_1, P_2 \in \mathcal{P}$ such that $V(P_1) \subseteq A\setminus B$ and $V(P_2) \subseteq B \setminus A$. Let $\mathcal{P}_1:=\{P \in \mathcal{P} : V(P) \subseteq A \setminus B\}$, $\mathcal{P}_2:=\{P \in \mathcal{P} : V(P) \subseteq B \setminus A\}$, and let us denote $k_1:=|\mathcal{P}_1|$, $k_2:=|\mathcal{P}_2|$. Note that since $(A,B)$ is a separation, every path $P \in \mathcal{P}\setminus (\mathcal{P}_1 \cup \mathcal{P}_2)$ must intersect the separator $A\cap B$. Since the paths in $\mathcal{P}$ are pairwise disjoint, this implies that $k-(k_1+k_2)\le |A\cap B|<2^{8r^8+1}$. Note that $\mathcal{P}_1$ is a collection of $k_1 \ge 1$ pairwise disjoint $(X\cap (A \setminus B))\dash (Y\cap (A\setminus B))$-paths in $G[A\setminus B]$, and $\mathcal{P}_2$ is a collection of $k_2 \ge 1$ pairwise disjoint $(X\cap (B \setminus A))\dash (Y\cap (B\setminus A))$-paths in $G[B\setminus A]$. Since by our minimality assumption on $G$ both graphs $G[A\setminus B]$ and $G[B\setminus A]$ satisfy the hypothesis of the theorem, we find that there is a collection $\mathcal{Q}_1$ of at least $\varepsilon k_1+\frac{1}{2}$ pairwise non-adjacent $(X\cap (A \setminus B))\dash (Y\cap (A\setminus B))$-paths in $G[A\setminus B]$, and a collection $\mathcal{Q}_2$ of at least $\varepsilon k_2+\frac{1}{2}$ pairwise non-adjacent $(X\cap (B \setminus A))\dash (Y\cap (B\setminus A))$-paths in $G[B\setminus A]$. Since there are no edges in $G$ between $A \setminus B$ and $B\setminus A$, the collection $\mathcal{Q}:=\mathcal{Q}_1\cup \mathcal{Q}_2$ also consists of pairwise non-adjacent $X\dash Y$-paths in $G$. We furthermore have
\begin{align*}
    |\mathcal{Q}|&=|\mathcal{Q}_1|+|\mathcal{Q}_2|\ge \left(\varepsilon k_1+\frac{1}{2}\right)+\left(\varepsilon k_2+\frac{1}{2}\right)=\varepsilon(k_1+k_2)+1\\
    &> \varepsilon\left(k-2^{8r^8+1}\right)+1=\varepsilon k+\frac{1}{2}+\left(\frac{1}{2}-\varepsilon 2^{8r^8+1}\right)>\varepsilon k+\frac{1}{2}.
\end{align*}

This is a contradiction on our initial assumptions that such a collection $\mathcal{Q}$ cannot exist. Hence our assumption was false, and this concludes the proof of the claim.
\renewcommand{\qedsymbol}{$\blacksquare$}
\end{proof}
Next, we apply Theorem~\ref{thm:grohemarx} to $G$, which yields a tree-decomposition $(T,(V_t)_{t \in V(T)})$ of $G$ of adhesion less than $r^2$, such that every torso $\tau(t)$ has at most $r^2$ vertices of degree at least $2r^4$, or is $K_h$-minor-free for $h:=2r^2$. 

\begin{claim}\label{claim:bagselect}
    There exists a vertex $t^\ast \in V(T)$ such that $V(P) \cap V_{t^\ast}\neq \emptyset$ for every $P \in \mathcal{P}$.
\end{claim}

\begin{proof}[Proof of \cref*{claim:bagselect}.]
For every edge $e=t_1t_2$ of $T$, we have that $$\left(\bigcup_{t \in (T-e)(t_1)}{V_t},\bigcup_{t \in (T-e)(t_2)}{V_t}\right)$$ forms a separation in $G$ of order $|S(e)|<r^2<2^{8r^8+1}$. Hence, by \cref{claim:separation}, every path in $\mathcal{P}$ intersects $\bigcup_{t \in (T-e)(t_1)}{V_t}$, or every path in $\mathcal{P}$ intersects $\bigcup_{t \in (T-e)(t_2)}{V_t}$. We can therefore find an orientation $\vec{T}$ of  $T$ such that for every edge $e=t_1t_2$ oriented from $t_1$ to $t_2$ in $\vec{T}$, we have $V(P) \cap \left(\bigcup_{t \in (T-e)(t_2)}{V_t}\right)\neq \emptyset$ for every $P \in \mathcal{P}$. Since $T$ is a tree, there must exists a vertex $t^\ast \in V(T)$ that is a sink in the orientation $\vec{T}$ of $T$. We now claim that $V(P) \cap V_{t^\ast}\neq\emptyset$ for every $P \in \mathcal{P}$. Suppose otherwise towards a contradiction. Let $P \in \mathcal{P}$ be such that $V(P) \cap V_{t^\ast}=\emptyset$ and let $R:=\{t \in V(T) : V(P) \cap V_t \neq \emptyset\}$. Since $P$ is a connected subgraph of $G$, it readily follows from the definition of a tree-decomposition that $R$ induces a subtree of $T$, which does not include $t^\ast$. Hence, there is an edge $e=t't^\ast \in E(T)$ incident to $t^\ast$ such that $R \subseteq (T-e)(t')$. This, however, contradicts the fact that $V(P) \cap \left(\bigcup_{t \in (T-e)(t^\ast)}{V_t}\right)\neq\emptyset$, which follows since $e$ is oriented from $t'$ to $t^\ast$ in $\vec{T}$. This concludes the proof of the claim.
\renewcommand{\qedsymbol}{$\blacksquare$}
\end{proof}

Let $H$ be the graph obtained from $G[V_{t^\ast}]$ by adding an edge between every pair $x,y$ of non-adjacent vertices in $G[V_{t^\ast}]$ for which there exists a path in $G$ with endpoints $x,y$ all whose internal vertices are in $V(G)\setminus V_{t^\ast}$. 

\begin{claim}\label{claim:subgraphs}For every pair of vertices $x,y \in V_{t^\ast}$ with $xy \notin E(G)$ for which there exists a path in $G$ with endpoints $x,y$ all whose internal vertices are in $V(G)\setminus V_{t^\ast}$, there exists an edge $f=tt^\ast \in E(T)$ incident with $t^\ast$ such that $x,y \in S(f)$. In particular, $G[V_{t^\ast}]\subseteq H\subseteq \tau(t^\ast)$.
\end{claim}
\begin{proof}[Proof of \cref*{claim:subgraphs}.]
Let $P$ be an $x\dash y$-path in $G$ such that $V(P) \cap V_{t^\ast}=\{x,y\}$. Let $S:=\{s \in V(T) : V_s \cap (V(P)\setminus \{x,y\}) \neq \emptyset\}$. It follows readily from the definition of a tree-decomposition (and since $P-\{x,y\}$ is a connected subgraph of $G$) that $S$ induces a connected subgraph of $T$, i.e., $T[S]$ is a subtree of $T$. We furthermore have $V_{t^\ast}\cap (V(P) \setminus \{x,y\})=\emptyset$, and thus $t^\ast \notin S$. Therefore, there exists an edge $f=tt^\ast$ incident with $t^\ast$ such that $S$ is contained in $(T-f)(t)$. Let $x,x_1,\ldots,x_\ell,y$ be the vertex-trace of $P$. By definition of a tree-decomposition, there exist bags $V_{t_1}$ and $V_{t_2}$ such that $x,x_2 \in V_{t_1}$ and $x_{\ell-1},y \in V_{t_2}$. This directly implies that $t_1,t_2 \in S\subseteq (T-f)(t)$. Hence, we have
$$x, y\in \left(\bigcup_{s \in (T-f)(t)}{V_s}\right) \cap V_{t^\ast}=V_t \cap V_{t^\ast}.$$ This proves that $x,y \in S(f)$, as desired. This concludes the proof.
\renewcommand{\qedsymbol}{$\blacksquare$}
\end{proof}

Next, we define a family $\mathcal{P}^\ast$ of $k$ disjoint paths in $H$ as follows. For every path $P \in \mathcal{P}$, let $P^\ast$ denote the path in $H$ that has vertex-set $V(P) \cap V_{t^\ast}$ and visits the vertices in $V(P) \cap V_{t^\ast}$ in the same order as $P$. This indeed forms a path in $H$, since for every subpath $x,x_1,\ldots,x_\ell,y$ of $P$ with $x, y \in V_{t^\ast}$ and $x_1,\ldots,x_\ell \notin V_{t^\ast}$, we have $xy \in E(H)$ by definition.

For every endpoint $v$ of a path $P \in \mathcal{P}$, let us denote by $v^\ast\in V_{t^\ast}$ the unique vertex in $V(P) \cap V_{t^\ast}$ that is closest to $v$ along the path $P$. Necessarily, $v^\ast$ is an endpoint of $P^\ast$. Using this notation, we now define two subsets $X^\ast, Y^\ast\subseteq V_{t^\ast}$ as $$X^\ast:=\{v^\ast:v \in X \text{ and }v \text{ is the endpoint of some path in }\mathcal{P}\},$$
$$Y^\ast:=\{v^\ast:v \in Y \text{ and }v \text{ is the endpoint of some path in }\mathcal{P}\}.$$In particular, $\mathcal{P}^\ast$ is a collection of $k$ pairwise disjoint $X^\ast\dash Y^\ast$-paths in $H$. 

\begin{claim}\label{claim:restrictedpaths}
    There exists a family $\mathcal{Q}^\ast$ consisting of pairwise non-adjacent $X^\ast\dash Y^\ast$-paths in $H$ such that $|\mathcal{Q}^\ast|\ge \varepsilon k + \frac{1}{2}$.
\end{claim}

\begin{proof}[Proof of \cref*{claim:restrictedpaths}.]
By the properties of the tree-decomposition $(T,(V_t)_{t \in V(T)})$, we know that $\tau(t^\ast)$ either has at most $r^2$ vertices of degree at least $2r^4$, or is $K_h$-minor-free for $h=2r^2$. In particular, the same is true for the subgraph $H$ of $\tau(t^\ast)$. 

Let us start with considering the first case. Let $Z \subseteq V_{t^\ast}$ denote the set of vertices of degree at least $2r^4$ in $H$. Then we know that $|Z|\le r^2$ and $\Delta(H-Z)<2r^4:=\Delta$. Let $\mathcal{P}'$ be the set of paths of ${\mathcal P}^\ast$ which do not intersect $Z$. Then, since the paths in $\mathcal{P}^\ast$ are pairwise disjoint, we have that $|\mathcal{P}^\ast|\ge k-|Z|\ge k-2r^4$.  Hence, \cref{thm:strongcolouring2} implies that there exists a collection $\mathcal{Q}^\ast$ of pairwise non-adjacent $(X^\ast\setminus Z)\dash (Y^\ast\setminus Z)$-paths in $H-Z$ such that 
$$|\mathcal{Q}^\ast|\ge \left\lfloor\frac{1}{2^{2\Delta^2}}(k-2r^4)\right\rfloor.$$
Clearly, $\mathcal{Q}^\ast$ is also a family of pairwise non-adjacent $X^\ast\dash Y^\ast$-paths in $H$. Using that $k>\frac{1}{2\varepsilon}$ and $\varepsilon=2^{-(8t^8+3)}$, we can lower bound its size as follows.
$$|\mathcal{Q}^\ast|\ge \frac{k}{2^{2\Delta^2}}-\frac{2r^4}{2^{2\Delta^2}}-1=\frac{k}{2^{8r^8}}-\frac{2r^4}{2^{8r^8}}-1=8\varepsilon k -\frac{2r^4}{2^{8r^8}}-1=\varepsilon k+\frac{1}{2} + \left(7\varepsilon k-\frac{2r^4}{2^{8r^8}}-\frac{3}{2}\right)>\varepsilon k+\frac{1}{2}.$$ This establishes the claim in the first case. 

Next, consider the case that $H$ is $K_h$-minor-free. Then, by Lemma~\ref{lemma:minorfree} there exists a collection $\mathcal{Q}^\ast$ of pairwise non-adjacent $X^\ast\dash Y^\ast$-paths in $H$ of size at least $\frac{k}{2(h-1)}=\frac{k}{2(2r^2-1)}>2^{-(8r^8+2)} k=2\varepsilon k>\varepsilon k+\frac{1}{2}$, as desired. This concludes the proof of the claim also in the second possible case.
\renewcommand{\qedsymbol}{$\blacksquare$}
\end{proof}

We now finish the proof of the theorem by using $\mathcal{Q}^\ast$, as given by \cref{claim:restrictedpaths}, to construct a family $\mathcal{Q}$ of pairwise non-adjacent $X\dash Y$-paths in $G$ of size $|\mathcal{Q}|=|\mathcal{Q}^\ast|\ge \varepsilon k +\frac{1}{2}$.

For every edge $xy\in E(H)\setminus E(G[V_{t^\ast}])$, pick and fix a path $P_{xy}$ in $G$ that has endpoints $x,y$ and no internal vertices in $V_{t^\ast}$ (such a path always exists by definition of $H$). Furthermore, for every edge $xy \in E(G[V_{t^\ast}])$, we let $P_{xy}$ denote the path consisting of the single edge $xy$ in $G$. 

Now consider any path $Q^\ast \in \mathcal{Q}^\ast$ and let $x_1,x_2,\ldots,x_q$ be its sequence of vertices, such that $x_1 \in X^\ast$ and $x_q \in Y^\ast$. Then, by definition, there exist $x \in X$, $y \in Y$ such that $x^\ast=x_1$, $y^\ast=x_q$. We now define $W(Q^\ast)$ as the walk in $G$ that starts at $x$, follows the unique path in $\mathcal{P}$ that $x$ is an endpoint of, until it reaches $x^\ast=x_1$, then follows the concatenation of the paths $P_{x_ix_{i+1}}$ for $1 \le i <q$ and then follows the unique path in $\mathcal{P}$ that $y$ is an endpoint of, until it reaches $y$.

\begin{claim}\label{claim:nonadjacent}
    If $Q_1^\ast, Q_2^\ast \in \mathcal{Q}^\ast$ are distinct, then $W(Q_1^\ast)$ and $W(Q_2^\ast)$ are non-adjacent in $G$.
\end{claim}

\begin{proof}[Proof of \cref*{claim:nonadjacent}.]
Suppose towards a contradiction that there exist $a\in V(W(Q_1^\ast)), b \in V(W(Q_2^\ast))$ that are at distance at most $1$ in $G$. Let $a' \in V(Q_1^\ast)$ be a vertex closest to $a$ along $W(Q_1^\ast)$, and let $b' \in V(Q_2^\ast)$ be defined similarly for $b$. In particular, there exist paths $R_1$ and $R_2$ that form subwalks of $V(W(Q_1^\ast))$ and $V(W(Q_2^\ast))$, respectively, such that $R_1$ has endpoints $a, a'$ and $V(R_1) \cap V_{t^\ast}=\{a'\}$, and analogously  $R_2$ has endpoints $b, b'$ and $V(R_2) \cap V_{t^\ast}=\{b'\}$. Now, as $a=b$ or $ab \in E(G)$, then the walk $W$ in $G$ that starts at $a'$, follows $R_1$ to $a$, moves to $b$, and follows $R_2$ until it reaches $b'$, satisfies $V(W)\cap V_{t^\ast}=\{a',b'\}$. This implies that there exists an $a'\dash b'$-path $R$ in $G$ with $V(R)\subseteq V(W)$, in particular we have $V(R)\cap V_{t^\ast}=\{a',b'\}$. By definition of $H$, this implies that $a'=b'$ or $a' b' \in E(H)$, in either case a contradiction, since $Q_1^\ast$ and $Q_2^\ast$ are be non-adjacent in $H$. This concludes the proof of the claim.
\renewcommand{\qedsymbol}{$\blacksquare$}
\end{proof}

We can now define $\mathcal{Q}$ by, for every $Q^\ast\in \mathcal Q^\ast$, short-cutting the walk $W(Q^\ast)$ into a path in $G$ that has the same endpoints and such that $V(Q)\subseteq V(W(Q^\ast))$. By \cref{claim:nonadjacent}, any two distinct paths in $\mathcal{Q}$ are non-adjacent. Clearly, $|\mathcal{Q}|=|\mathcal{Q}^\ast|$ by definition, and \cref{claim:restrictedpaths} now implies that $\mathcal{Q}$ consists of at least $\varepsilon k+\frac{1}{2}$ pairwise non-adjacent $X\dash Y$-paths in $G$. This yields the desired contradiction, completing the proof of the theorem.
\end{proof}

\section{Subcubic graphs}\label{sec:subcubic}

In this section, we show our results on subcubic graphs. We begin by proving \cref{thm:kpathssubcubic}, which we restate for convenience.

\kpathssubcubic*

\begin{proof}
    Let $F:=E(G[V(\mathcal P)])\setminus E(\mathcal P)$. By \cref{thm:strongcolouring3}, it suffices to partition $F$ into four induced matchings. By hypothesis, no two edges of $F$ share an end vertex.

    A standard tool for studying strong edge colouring is to find a proper vertex colouring of the square of the line graph; we slightly vary this argument given that we only want to partition the edges in $F$. We construct the auxiliary graph $H$ with vertex set $F$ and such that $e_1,e_2\in F$ are adjacent if $\dist_G(e_1,e_2)=1$. In other words, $e_1,e_2$ are adjacent if and only if one of the end vertices of $e_1$ and one of the end vertices of $e_2$ are consecutive vertices on one of the paths in $\mathcal P$. This implies that $\Delta(H)\leq 4$, since the maximum degree in a path is two and since no two edges in $F$ are incident.
    
    By construction, a proper vertex-colouring of $H$ with four colours yields the desired partition; each colour class is an induced matching. It suffices to show that $H$ is $3$-degenerate. Suppose, for a contradiction, that $H$ has a $4$-regular subgraph $H'$. Let $P \in \mc{P}$ such that some edge  $e \in V(H')$ has an end  $u \in V(P)$ and choose such $e$ and $u$ so that the distance from $u$ to an end of $P$ along $P$ is minimum. Then at most one neighbour of $u$ in $P$ is an end of an edge in $V(H')$. It follows that $e$ has degree at most three in $H'$, obtaining the desired contradiction to the assumption that $H'$ is $4$-regular.

\end{proof}

We note that the $4$-colouring of the auxiliary graph $H$ in the proof above is best possible. The configuration shown in \cref{fig:4coloursexample} (where the horizontal edges are part of the paths of $\mathcal P$), is an example in which we cannot partition the edges outside $\mathcal P$ into three induced matchings. In particular, in this case the auxiliary graph is the Moser spindle \cite{moser_problems_1961}, which is easily verified to not be $3$-colourable.

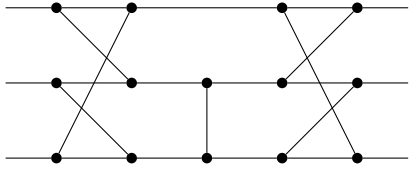
\begin{figure}
         \centering
         \begin{tikzpicture}[scale=1,
			dot/.style = {circle, fill, minimum size=#1,
			inner sep=0pt, outer sep=0pt},
			dot/.default = 4pt]

            \node (A0) at (0.2,3) {};
            \node (B0) at (0.2,2) {};
            \node (C0) at (0.2,1) {};
            
            \node[dot] (A1) at (1,3) {};
            \node[dot] (B1) at (1,2) {};
            \node[dot] (C1) at (1,1) {};

            \node[dot] (A2) at (2,3) {};
            \node[dot] (B2) at (2,2) {};
            \node[dot] (C2) at (2,1) {};

            \node[dot] (B3) at (3,2) {};
            \node[dot] (C3) at (3,1) {};

            \node[dot] (A4) at (4,3) {};
            \node[dot] (B4) at (4,2) {};
            \node[dot] (C4) at (4,1) {};

            \node[dot] (A5) at (5,3) {};
            \node[dot] (B5) at (5,2) {};
            \node[dot] (C5) at (5,1) {};

            \node (A6) at (5.8,3) {};
            \node (B6) at (5.8,2) {};
            \node (C6) at (5.8,1) {};

            \draw (A0) to (A1);
            \draw (A1) to (A2);
            \draw (A2) to (A4);
            \draw (A4) to (A5);
            \draw (A5) to (A6);

            \draw (B0) to (B1);
            \draw (B1) to (B2);
            \draw (B2) to (B3);
            \draw (B3) to (B4);
            \draw (B4) to (B5);
            \draw (B5) to (B6);

            \draw (C0) to (C1);
            \draw (C1) to (C2);
            \draw (C2) to (C3);
            \draw (C3) to (C4);
            \draw (C4) to (C5);
            \draw (C5) to (C6);

            \draw (C1) to (A2);
            \draw (A1) to (B2);
            \draw (B1) to (C2);

            \draw (B3) to (C3);

            \draw (A4) to (C5);
            \draw (C4) to (B5);
            \draw (B4) to (A5);

		\end{tikzpicture}
        \caption{Example requiring  four colours for any strong edge colouring of non-horizontal edges.}
        \label{fig:4coloursexample}
    \end{figure}

We now prove \cref{thm:twopathssubcubic}, which we restate for convenience.

\twopathssubcubic*

\begin{proof}
    We begin by proving the first part of the statement.

    We define a \emph{path system} as a quadruple $\mathcal H=(H,A,B,\mathcal Q)$ where $H$ is a graph, $A,B\subseteq V(H)$ and $\mathcal Q$ is a $5$-tuple of five disjoint $A\dash B$-paths such that
    \begin{enumerate}
        \item $V(H)=V(\mathcal Q)$,
        \item no vertex in $A$ is incident in $H$ to any edge in $E(H)\setminus E(\mathcal Q)$,
        \item every vertex of $V(H)\setminus A$ is incident in $H$ to exactly one edge in $E(H)\setminus E(\mathcal Q)$, and
        \item there does not exist any collection $\mathcal Q'$ of $5$ pairwise disjoint $A\dash B$-paths such that $V(\mathcal Q')\subsetneq V(H)$.
    \end{enumerate}
    Note that conditions (1) and (4) imply that the paths in $\mathcal Q$ contain no vertices in $A\cup B$ other than their endpoints, and if $x\in A\cap B$ then one of the paths consists of exactly $x$. In particular $|A|=|B|=5$.

    We may suppose without loss of generality that $\mathcal G=(G,X,Y,\mathcal P)$ is a path system. 
    By restricting $G$, $X$ and $Y$ to vertices $V(\mathcal P)$, we may suppose that (1) holds; any pair of non-adjacent paths in an induced subgraph remains non-adjacent in the original graph. We may suppose (2) holds as if a vertex $x\in X$ is incident to some edge in $E(H)\setminus E(\mathcal Q)$, we can add a new vertex $x'$ to $G$ as well as the edge $x'x$, replace $X$ by $(X\setminus \{x\})\cup \{x'\}$, and prepend $x'x$ to the path of $\mathcal P$ with $x$ as an endpoint. Any $x'\dash Y$-path in the new graph directly yields a $x\dash Y$-path with the same set of neighbours. We may suppose (3) holds for vertices not in $Y$ given that finding a pair of non-adjacent paths in a graph directly yields such a pair in a subdivision of this graph. Furthermore, if some vertex $y\in Y\setminus X$ is not incident to some edge in $E(G)\setminus E(\mathcal P)$, then $y$ has a neighbour in the path of $\mathcal P$ of which it is an endpoint, which exists by (1) (this path cannot be a singleton as $y\notin X$); say $y'y\in E(\mathcal P)$. We may then remove $y$ from $G$ and from its path of $\mathcal P$ and replace $Y$ by $(Y\setminus\{y\})\cup\{y'\}$; any $X\dash y'$-path directly extends to a $X\dash y$-path with the same neighbours. By repeating this argument, we may suppose that (3) holds for vertices in $Y$. Finally, we may suppose (4), as otherwise we could then replace $\mathcal P$ with these paths. It is easily verified that none of these reductions are in conflict.

    In the following, for $k\in \N$ we write $[k]=\{1,\dots,k\}$.
    
    We now define an operation (which will have two variants) which allows us to easily construct and represent path systems. Let $\mathcal H=(H,A,B,\mathcal Q=(Q_1,Q_2,Q_3,Q_4,Q_5))$ be a path system. For $i\in [5]$, write $b_i$ for the vertex of $B$ in $Q_i$.

    Let $\mathcal E:=\binom{[5]}{2}$ be the collection of (unordered) pairs of integers between $1$ and $5$. Let $\{i_1,i_2\}\in \mathcal E$. We define $\mathcal H\oplus \{i_1,i_2\}$ as the path system obtained by
    \begin{itemize}
        \item adding new vertices $b_{i_1}',b_{i_2}'$ and the edges $b_{i_1}b_{i_1}'$, $b_{i_2}b_{i_2}'$ and $b_{i_1}'b_{i_2}'$ to $H$,
        \item appending the edges $b_{i_1}b_{i_1}'$ and $b_{i_2}b_{i_2}'$, respectively, to $Q_{i_1}$ and $Q_{i_2}$, and
        \item replacing $B$ with $(B\setminus \{b_{i_1},b_{i_2}\})\cup \{b_{i_1}',b_{i_2}'\}$
    \end{itemize}
    An example of this operation is provided in \cref{fig:addexample}(b).
    
    Let $\mathcal C$ be the set of cyclic permutations of length at least two with values in $[5]$. We write such cycles as $(i_1 \dots i_k)$, for instance $(i_1 \ i_2 \ i_3)=(i_2 \ i_3 \ i_1)=(i_3 \ i_1 \ i_2)$. For $(i_1\ \dots \ i_k)\in \mathcal C$, we say $\mathcal H\oplus (i_1 \ \dots \  i_k)$ is the path system obtained by
    \begin{itemize}
        \item adding new vertices $c_{i_1},\dots,c_{i_k}$ and $b_{i_1}',\dots,b_{i_k}'$ and the edges $b_{i_1}c_{i_1},\dots,b_{i_k}c_{i_k}$, $c_{i_1}b_{i_1}',\dots,c_{i_k}b_{i_k}'$ and $c_{i_1}b_{i_2}',\dots,c_{i_k}b_{i_1}'$ to $H$,
        \item appending edges $b_{i_j}c_{i_j}$ and $c_{i_j}b_{i_j}'$ to $Q_{i_j}$ for every $j\in [k]$, and 
        \item replacing $B$ with $(B\setminus \{b_{i_1},\dots,b_{i_k}\})\cup \{b_{i_1}',\dots,b_{i_k}'\}$.
    \end{itemize}
    An example is provided in \cref{fig:addexample}(c).

    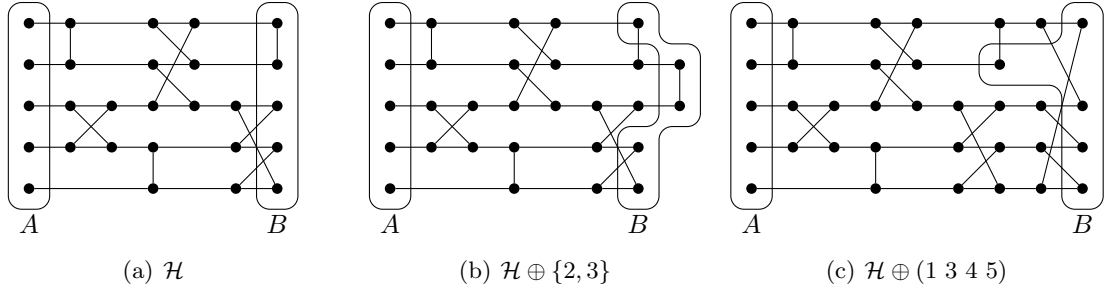
\begin{figure}
        \begin{subfigure}[t]{.32\textwidth}
         \centering
         \begin{tikzpicture}[scale=0.55,
			dot/.style = {circle, fill, minimum size=#1,
			inner sep=0pt, outer sep=0pt},
			dot/.default = 4pt]

            \draw[rounded corners=5pt] (-0.5, 0.5) rectangle (0.5, 5.5) {};
            \node at (0,0.15) {$A$};

            \draw[rounded corners=5pt] (5.5, 0.5) rectangle (6.5, 5.5) {};
            \node at (6,0.15) {$B$};
            
            \node[dot] (A0) at (0,5) {};
            \node[dot] (B0) at (0,4) {};
            \node[dot] (C0) at (0,3) {};
            \node[dot] (D0) at (0,2) {};
            \node[dot] (E0) at (0,1) {};

            \node[dot] (A1) at (1,5) {};
            \node[dot] (B1) at (1,4) {};
            \node[dot] (C1) at (1,3) {};
            \node[dot] (D1) at (1,2) {};

            \node[dot] (C2) at (2,3) {};
            \node[dot] (D2) at (2,2) {};

            \node[dot] (A3) at (3,5) {};
            \node[dot] (B3) at (3,4) {};
            \node[dot] (C3) at (3,3) {};
            \node[dot] (D3) at (3,2) {};
            \node[dot] (E3) at (3,1) {};

            \node[dot] (A4) at (4,5) {};
            \node[dot] (B4) at (4,4) {};
            \node[dot] (C4) at (4,3) {};

            \node[dot] (C5) at (5,3) {};
            \node[dot] (D5) at (5,2) {};
            \node[dot] (E5) at (5,1) {};

            \node[dot] (A6) at (6,5) {};
            \node[dot] (B6) at (6,4) {};
            \node[dot] (C6) at (6,3) {};
            \node[dot] (D6) at (6,2) {};
            \node[dot] (E6) at (6,1) {};

            \draw (A1) to (B1);
            \draw (C1) to (D2);
            \draw (D1) to (C2);

            \draw (D3) to (E3);

            \draw (A3) to (B4);
            \draw (B3) to (C4);
            \draw (C3) to (A4);

            \draw (C6) to (D5);
            \draw (D6) to (E5);
            \draw (E6) to (C5);

            \draw (A6) to (B6);

            \draw (A0) to (A1);
            \draw (A1) to (A3);
            \draw (A3) to (A4);
            \draw (A4) to (A6);

            \draw (B0) to (B1);
            \draw (B1) to (B3);
            \draw (B3) to (B4);
            \draw (B4) to (B6);

            \draw (C0) to (C1);
            \draw (C1) to (C2);
            \draw (C2) to (C3);
            \draw (C3) to (C4);
            \draw (C4) to (C5);
            \draw (C5) to (C6);

            \draw (D0) to (D1);
            \draw (D1) to (D2);
            \draw (D2) to (D3);
            \draw (D3) to (D5);
            \draw (D5) to (D6);

            \draw (E0) to (E3);
            \draw (E3) to (E5);
            \draw (E5) to (E6);

		\end{tikzpicture}
        \caption{$\mathcal H$}
        \end{subfigure}
        \begin{subfigure}[t]{.32\textwidth}
         \centering
         \begin{tikzpicture}[scale=0.55,
			dot/.style = {circle, fill, minimum size=#1,
			inner sep=0pt, outer sep=0pt},
			dot/.default = 4pt]

            \draw[rounded corners=5pt] (-0.5, 0.5) rectangle (0.5, 5.5) {};
            \node at (0,0.15) {$A$};

            \draw[rounded corners=5pt] (5.5,0.5) -- (5.5,2.5) -- (6.5,2.5) -- (6.5,4.5) -- (5.5,4.5) -- (5.5,5.5) -- (6.5,5.5) -- (6.5,4.5) -- (7.5,4.5) -- (7.5,2.5) -- (6.5,2.5) -- (6.5,0.5) -- cycle {};
            \node at (6,0.15) {$B$};
            
            \node[dot] (A0) at (0,5) {};
            \node[dot] (B0) at (0,4) {};
            \node[dot] (C0) at (0,3) {};
            \node[dot] (D0) at (0,2) {};
            \node[dot] (E0) at (0,1) {};

            \node[dot] (A1) at (1,5) {};
            \node[dot] (B1) at (1,4) {};
            \node[dot] (C1) at (1,3) {};
            \node[dot] (D1) at (1,2) {};

            \node[dot] (C2) at (2,3) {};
            \node[dot] (D2) at (2,2) {};

            \node[dot] (A3) at (3,5) {};
            \node[dot] (B3) at (3,4) {};
            \node[dot] (C3) at (3,3) {};
            \node[dot] (D3) at (3,2) {};
            \node[dot] (E3) at (3,1) {};

            \node[dot] (A4) at (4,5) {};
            \node[dot] (B4) at (4,4) {};
            \node[dot] (C4) at (4,3) {};

            \node[dot] (C5) at (5,3) {};
            \node[dot] (D5) at (5,2) {};
            \node[dot] (E5) at (5,1) {};

            \node[dot] (A6) at (6,5) {};
            \node[dot] (B6) at (6,4) {};
            \node[dot] (C6) at (6,3) {};
            \node[dot] (D6) at (6,2) {};
            \node[dot] (E6) at (6,1) {};

            \node[dot] (B7) at (7,4) {};
            \node[dot] (C7) at (7,3) {};

            \draw (A1) to (B1);
            \draw (C1) to (D2);
            \draw (D1) to (C2);

            \draw (D3) to (E3);

            \draw (A3) to (B4);
            \draw (B3) to (C4);
            \draw (C3) to (A4);

            \draw (C6) to (D5);
            \draw (D6) to (E5);
            \draw (E6) to (C5);

            \draw (A6) to (B6);

            \draw (A0) to (A1);
            \draw (A1) to (A3);
            \draw (A3) to (A4);
            \draw (A4) to (A6);

            \draw (B0) to (B1);
            \draw (B1) to (B3);
            \draw (B3) to (B4);
            \draw (B4) to (B6);

            \draw (C0) to (C1);
            \draw (C1) to (C2);
            \draw (C2) to (C3);
            \draw (C3) to (C4);
            \draw (C4) to (C5);
            \draw (C5) to (C6);

            \draw (D0) to (D1);
            \draw (D1) to (D2);
            \draw (D2) to (D3);
            \draw (D3) to (D5);
            \draw (D5) to (D6);

            \draw (E0) to (E3);
            \draw (E3) to (E5);
            \draw (E5) to (E6);

            \draw (B7) to (C7);
            \draw (B6) to (B7);
            \draw (C6) to (C7);

		\end{tikzpicture}
        \caption{$\mathcal H\oplus\{2,3\}$}
        \end{subfigure}
        \begin{subfigure}[t]{.32\textwidth}
         \centering
         \begin{tikzpicture}[scale=0.55,
			dot/.style = {circle, fill, minimum size=#1,
			inner sep=0pt, outer sep=0pt},
			dot/.default = 4pt]

            \draw[rounded corners=5pt] (-0.5, 0.5) rectangle (0.5, 5.5) {};
            \node at (0,0.15) {$A$};

            \draw[rounded corners=5pt] (7.5,0.5) -- (7.5,3.5) -- (5.5,3.5) -- (5.5,4.5) -- (7.5,4.5) -- (7.5,5.5) -- (8.5,5.5) -- (8.5,0.5) -- cycle {};
            \node at (8,0.15) {$B$};
            
            \node[dot] (A0) at (0,5) {};
            \node[dot] (B0) at (0,4) {};
            \node[dot] (C0) at (0,3) {};
            \node[dot] (D0) at (0,2) {};
            \node[dot] (E0) at (0,1) {};

            \node[dot] (A1) at (1,5) {};
            \node[dot] (B1) at (1,4) {};
            \node[dot] (C1) at (1,3) {};
            \node[dot] (D1) at (1,2) {};

            \node[dot] (C2) at (2,3) {};
            \node[dot] (D2) at (2,2) {};

            \node[dot] (A3) at (3,5) {};
            \node[dot] (B3) at (3,4) {};
            \node[dot] (C3) at (3,3) {};
            \node[dot] (D3) at (3,2) {};
            \node[dot] (E3) at (3,1) {};

            \node[dot] (A4) at (4,5) {};
            \node[dot] (B4) at (4,4) {};
            \node[dot] (C4) at (4,3) {};

            \node[dot] (C5) at (5,3) {};
            \node[dot] (D5) at (5,2) {};
            \node[dot] (E5) at (5,1) {};

            \node[dot] (A6) at (6,5) {};
            \node[dot] (B6) at (6,4) {};
            \node[dot] (C6) at (6,3) {};
            \node[dot] (D6) at (6,2) {};
            \node[dot] (E6) at (6,1) {};

            \node[dot] (A7) at (7,5) {};
            \node[dot] (C7) at (7,3) {};
            \node[dot] (D7) at (7,2) {};
            \node[dot] (E7) at (7,1) {};
            
            \node[dot] (A8) at (8,5) {};
            \node[dot] (C8) at (8,3) {};
            \node[dot] (D8) at (8,2) {};
            \node[dot] (E8) at (8,1) {};

            \draw (A1) to (B1);
            \draw (C1) to (D2);
            \draw (D1) to (C2);

            \draw (D3) to (E3);

            \draw (A3) to (B4);
            \draw (B3) to (C4);
            \draw (C3) to (A4);

            \draw (C6) to (D5);
            \draw (D6) to (E5);
            \draw (E6) to (C5);

            \draw (A6) to (B6);

            \draw (A0) to (A1);
            \draw (A1) to (A3);
            \draw (A3) to (A4);
            \draw (A4) to (A6);

            \draw (B0) to (B1);
            \draw (B1) to (B3);
            \draw (B3) to (B4);
            \draw (B4) to (B6);

            \draw (C0) to (C1);
            \draw (C1) to (C2);
            \draw (C2) to (C3);
            \draw (C3) to (C4);
            \draw (C4) to (C5);
            \draw (C5) to (C6);

            \draw (D0) to (D1);
            \draw (D1) to (D2);
            \draw (D2) to (D3);
            \draw (D3) to (D5);
            \draw (D5) to (D6);

            \draw (E0) to (E3);
            \draw (E3) to (E5);
            \draw (E5) to (E6);
            
            \draw (A6) to (A7);
            \draw (C6) to (C7);
            \draw (D6) to (D7);
            \draw (E6) to (E7);

            \draw (A8) to (A7);
            \draw (C8) to (C7);
            \draw (D8) to (D7);
            \draw (E8) to (E7);

            \draw (A7) to (C8);
            \draw (C7) to (D8);
            \draw (D7) to (E8);
            \draw (E7) to (A8);

		\end{tikzpicture}
        \caption{$\mathcal H\oplus (1 \ 3 \ 4 \ 5)$}
        \end{subfigure}

        \caption{Example of a path system $\mathcal H$ and two examples for the $\oplus$ operation. The paths are labelled from 1 to 5 from top to bottom.}
        \label{fig:addexample}
    \end{figure}

    Let $\mathcal H_0=(H_0,A_0,B_0,\mathcal Q_0)$ be the path system consisting of the graph $H_0$ with singleton vertices $V(H_0)=A_0=B_0=\{v_1,v_2,v_3,v_4,v_5\}$ and $\mathcal P_0$ the 5 paths of length $0$ in $H_0$.
    
    We now show that every path system can be obtained from $\mathcal H_0$ using the $\oplus$ operation. We say two path systems $\mathcal H_1=(H_1,A_1,B_1,\mathcal Q_1)$ and $\mathcal H_2= (H_2,A_2,B_2,\mathcal Q_2)$ are \emph{isomorphic}, which we denote by $\mathcal H_1\simeq \mathcal H_2$, if there exists a graph isomorphism $h:V(H_1)\rightarrow (H_2)$ which maps $A_1$ to $A_2$, $B_1$ to $B_2$ and $\mathcal Q_1$ to $\mathcal Q_2$ (the ordering of the paths must be the same).

    \begin{claim}\label{claim:subcubicoplus}
        For every path system $\mathcal H$, there exists some sequence $m_1,\dots,m_k\in \mathcal E\cup \mathcal C$ such that $\mathcal H\simeq\mathcal H_0\oplus m_1\oplus \dots \oplus m_k$.
    \end{claim}
    
    \begin{proof}[Proof of \cref*{claim:subcubicoplus}.]
        Write $\mathcal H=(H,A,B,\mathcal Q)$. We prove the statement by induction on $|E(H)\setminus E(\mathcal Q)|$.
        
        For the base case, if $|E(H)\setminus E(\mathcal Q)|=0$, then (3) implies that $A=V(\mathcal Q)$, and so necessarily $\mathcal H\simeq \mathcal H_0$.

        We now show the inductive step. Let $F$ be the set of edges in $E(H)\setminus E(\mathcal Q)$ incident to $B$. As $|E(H)\setminus E(\mathcal Q)|>0$, condition (2) implies that $A\neq B$, and so by (3) we have that $F\neq \emptyset$.  

        First suppose there exists an edge of $F$ with both ends in $B$, say $b_{i_1}'b_{i_2}'$, where $b_{i_1}'\in V(Q_{i_1})$ and $b_{i_2}'\in V(Q_{i_2})$. Let $\mathcal H'$ be the path system resulting from removing $b_{i_1}',b_{i_2}'$. More precisely, if $b_{i_1}$ and $b_{i_2}$ are, respectively, the neighbours of $b_{i_1}'$ in $Q_{i_1}$ and of $b_{i_2}'$ in $Q_{i_2}$ (these exist since $b_{i_1}',b_{i_2}'\notin A$), then $\mathcal H'=(H',A,B',\mathcal Q')$ where $H'=H\left[V(H)\setminus \{ b_{i_1}',b_{i_2}'\}\right]$, $B'=\left(B\setminus\{b_{i_1}',b_{i_2}'\}\right)\cup\{b_{i_1},b_{i_2}\}$ and $\mathcal Q'$ is identical to $\mathcal Q$ except that the edges $b_{i_1}b_{i_1}'$ and $b_{i_2}b_{i_2}'$ are removed from, respectively, $Q_{i_1}$ and $Q_{i_2}$. It is direct from the definitions that $\mathcal H=\mathcal H'\oplus \{i_1,i_2\}$. In particular, condition (3) implies that $b_{i_1}'b_{i_2}'$ were not incident to any edge other than $b_{i_1}b_{i_1}'$, $b_{i_2}b_{i_2}'$ and $b_{i_1}'b_{i_2}'$. Furthermore, $E(H')\setminus E(\mathcal Q')=(E(H)\setminus E(\mathcal Q))\setminus \{b_{i_1}'b_{i_2}'\}$. By induction, there exists a sequence $m_1,\dots,m_{k-1}\in \mathcal E\cup \mathcal C$ such that $\mathcal H
        '\simeq \mathcal H_0\oplus m_1 \oplus \dots \oplus m_{k-1}$. Hence, we obtain that $\mathcal H\simeq \mathcal H_0\oplus m_1 \oplus \dots \oplus m_{k-1}\oplus\{i_1,i_2\}$, as desired.

        Otherwise, we construct an auxiliary digraph $J$ with vertex set $F$ as follows. Let $c_{i_1}b_{i_2}', c_{i_3}b_{i_4}'\in F$, where $b_{i_2}'\in V(Q_{i_2})\cap B$ and $b_{i_4}'\in V(Q_{i_4})\cap B$, and $c_{i_1}\in V(Q_{i_1})\setminus B$ and $c_{i_3}\in V(Q_{i_3})\setminus B$. Further note that $i_1\neq i_2$ and $i_3\neq i_4$: these edges are not in $\mathcal Q$ and by (4) the paths of $\mathcal Q$ are necessarily induced. In our auxiliary digraph $J$, we put a directed edge from $c_{i_1}b_{i_2}$ to $c_{i_3}b_{i_4}$ in $J$ if and only if $i_2=i_3$.

        Every edge in $F$ has in-degree at least one in $J$ (in fact, it is necessarily exactly one). Indeed, let $c_{i_1}b_{i_2}'\in F$. Let $b_{i_1}'$ be the unique vertex of $V(Q_{i_1})\cap B$. As $c_{i_1}\notin B$, necessarily $b_{i_1}'\notin A$. Hence, by (3), there exists there is some edge $e\in F$ incident to $b_{i_1}'$. In particular, there is a directed edge in $J$ from $e$ to $c_{i_1}b_{i_2}'$.
        
        Hence, there exists in $J$ a directed cycle. By definition, the sequence of vertices in this directed cycle is of the form $c_{i_1}b_{i_2}',\dots,c_{i_k}b_{i_1}'$, where $c_{i_j}\in V(Q_{i_j})\setminus B$ and $b_{i_j}'\in V(Q_{i_j})\cap B$ for every $j\in [k]$. We claim $c_{i_j}b_{i_j}'\in E(Q_{i_j})$ for every $j\in [k]$. Suppose otherwise that for some $j\in [k]$ there exists at least one vertex $x$ between $c_{i_j}$ and $b_{i_j}$ on $Q_{i_j}$. Then, the following collection of paths would contradict (4): in $\mathcal Q$, replace the paths $Q_{i_1},\dots,Q_{i_k}$ with the paths formed by following $P_{i_j}$ until $c_{i_j}$ and then following the edge $c_{i_j}b_{i_{j+1}}'$ (with addition modulo $k$). In particular, this new set of paths does not contain $x$.

        Let $\mathcal H'$ be the path system resulting from removing $c_{i_j}$ and $b_{i_j}'$ for every $j\in [k]$. More precisely, if we write $b_{i_j}$ for the neighbour of $c_{i_j}$ in $V(Q_{i_j})$ which is not $b_{i_j}'$ (this vertex exists since $c_{i_j}\notin A$ by (2)), for every $j\in [k]$, then $\mathcal H'=(H',A,B',\mathcal Q')$ where where $H'=H\left[V(H)\setminus \{ c_{i_1},\dots,c_{i_k},b_{i_1}',\dots,b_{i_k}'\}\right]$, $B'=\left(B\setminus\{b_{i_1}',\dots,b_{i_k}'\}\right)\cup\{b_{i_1},\dots,b_{i_k}\}$ and $\mathcal Q'$ is identical to $\mathcal Q$ except that the edges $b_{i_j}c_{i_j}$ and $c_{i_j}b_{i_j}'$ are removed from $Q_{i_j}$, for every $j\in [k]$. Similarly to above, it is then direct from the definitions that $\mathcal H\simeq \mathcal H_0\oplus m_1 \oplus \dots \oplus m_{k-1}\oplus(i_1\ \dots \ i_k))$, and by induction, there exists a sequence $m_1,\dots,m_{k-1}\in \mathcal E\cup \mathcal C$ such that $\mathcal H
        '\simeq \mathcal H_0\oplus m_1 \oplus \dots \oplus m_{k-1}$. Hence, we obtain that $\mathcal H\simeq \mathcal H_0\oplus m_1 \oplus \dots \oplus m_{k-1}\oplus(i_1\ \dots \ i_k)$, as desired.
        \renewcommand{\qedsymbol}{$\blacksquare$}
    \end{proof}

    We define a \emph{state} as an unordered pair $\{S_1,S_2\}$ of non-empty disjoint subsets of $[5]$. Given a path system $\mathcal H=(H,A,B,\mathcal Q=(Q_1,Q_2,Q_3,Q_4,Q_5))$, we say a state $S=\{S_1,S_2\}$ is \emph{$\mathcal H$-reachable} if there exist sets $C_1,C_2\subseteq V(H)$ such that
    \begin{itemize}
        \item $C_1,C_2$ are disjoint and non-adjacent in $H$,
        \item $C_1\cap A\neq \emptyset$ and $C_2\cap A\neq \emptyset$, 
        \item $C_1\cap V(Q_i)\cap B \neq \emptyset$ if and only if $i\in S_1$ and $C_2\cap V(Q_i)\cap B \neq \emptyset$ if and only if $i\in S_2$, and
        \item $H[C_1],H[C_2]$ are connected.
    \end{itemize}

    \begin{claim}\label{claim:subcubicreachable}
        For every path system $\mathcal H$, there exists an $\mathcal H$-reachable state.
    \end{claim}

    \begin{proof}[Proof of \cref*{claim:subcubicreachable}.]
        Let $S=\{S_1,S_2\}$ be a state and $m\in \mathcal E\cup \mathcal C$. We construct a collection of states $f(S,m)$ by saying $S'\in f(S,m)$ if and only if $S'$ can be written as $S'=\{S_1,S_2\}$ such that the following holds: writing, $\mathcal H_0\oplus m=(H_m,A_m,B_m,\mathcal Q_m)$, there there exists $C_1,C_2\subseteq V(H_m)$ such that
        \begin{itemize}
            \item $C_1,C_2$ are disjoint and non-adjacent in $H_m$,
            \item $C_1\cap V(Q_i)\cap A_m$ if and only if $i\in S_1$ and $C_2\cap V(Q_i)\cap A_m$ if and only if $i\in S_2$,
            \item $C_1\cap V(Q_i)\cap B_m \neq \emptyset$ if and only if $i\in S_1'$ and $C_2\cap V(Q_i)\cap B_m \neq \emptyset$ if and only if $i\in S_2'$,
            \item for every vertex of $C_1\cap B_m$, its connected component in $H[C_1]$ contains a vertex of $C_1\cap A_m$, and for every vertex of $C_2\cap B_m$, its connected component in $H[C_2]$ contains a vertex of $C_2\cap A_m$.
        \end{itemize}

        It is easily verified that for any path system $\mathcal H$, if $S$ is $\mathcal H$-reachable, then every state in $f(S,m)$ is $\mathcal H\oplus m$-reachable. The crucial observation is that $\mathcal H\oplus m$ can be obtained from $\mathcal H$ and $\mathcal H_0\oplus m$ by identifying the vertices in the $B$ set from the former and the vertices from the $A$ set in the latter. From this, these conditions are exactly those which allow us to extend the non-adjacent sets $C_1,C_2$ corresponding to states in $\mathcal H$ to sets in $\mathcal H\oplus m$.

        If $\mathcal S$ is a collection of states, let $g(\mathcal S,m)=\bigcup_{S\in \mathcal S}f(\mathcal S,m)$. Hence, for any $\mathcal H$, if $\mathcal S$ is a collection of $\mathcal H$-reachable states, then $g(\mathcal S,m)$ is a collection of $\mathcal H\oplus m$-reachable states. We say a collection of states $\mathcal S'$ is a \emph{descendant} of $\mathcal S$ if there exists some sequence $m_1,\dots,m_k\in \mathcal E\cup \mathcal C$ such that $\mathcal S'=g(\dots g(g(\mathcal S,m_1),m_2)\dots,m_k)$.

        Let $\mathcal S_0:=\{\{\{x\},\{y\}\}:x,y \in [5], x \neq y\}$. It is direct from the definition that every state in $\mathcal S_0$ is $\mathcal H_0$-reachable. By \cref{claim:subcubicoplus}, every $\mathcal H$ can be written as $\mathcal H=\mathcal G_0\oplus m_1\oplus\dots\oplus m_k$ for some sequence of $m_1,\dots,m_k\in \mathcal E\cup\mathcal C$. In particular, $g(\dots g(g(\mathcal S_0,m_1),m_2)\dots,m_k)$ is a collection of $\mathcal H$-reachable states. We want to show that this collection is non-empty. In order to prove the claim, it thus suffices to show that $\emptyset$ is not a descendant of $\mathcal S_0$. Our strategy is thus as follows : start with the collection $\mathcal S_0$, and then repeatedly apply $g(\cdot,m)$ for every $m\in \mathcal E\cup \mathcal C$ until no new collection of states are found. If the $\emptyset$ is never encountered, we are done.

        Along the way, we may in fact trim some branches off of this process, in the three following ways:
        \begin{enumerate}[label=(\roman*)]
            \item If $\mathcal S$ is a collection of $\mathcal H$-reachable states and $\sigma$ is a permutation of $[5]$, then $\sigma(\mathcal S)$, which is obtained by applying $\sigma$ to the elements in $\mathcal S$, is $\mathcal H'$-reachable, where $\mathcal H'$ is identical to $\mathcal H'$ except that  we have permuted the order of the paths of $\mathcal Q$ according to $\sigma$.
            
             \noindent If $g(\dots g(g(\mathcal \sigma (S),m_1),m_2)\dots,m_k)=\emptyset$, then $g(\dots g(g(\mathcal S',\sigma^{-1}(m_1)),\sigma^{-1}(m_2))\dots,\sigma^{-1}(m_k))$ is empty, and so if $\emptyset$ is not a descendant of $\mathcal S$, it is not a descendant of $\sigma(\mathcal S)$ either. In particular, we need only keep one collection of states from each equivalence class under permutation.
            \item If $\mathcal S\subseteq \mathcal S'$ are collections of states, then $g(\mathcal S,m)\subseteq g(\mathcal S',m)$. In particular, if $\emptyset$ is not a descendant of $\mathcal S$, it is not a descendent of $\mathcal S'$ either. Hence, during our searching process, we may throw out any collection of states which is not minimal.
            \item We do not need to use the definition of $f$ to compute every instance of $f(S,m)$. Indeed, similarly to (i), it is easily verified from the definitions that if $\sigma$ is a permutation of $[5]$, then $f (S, m)=\sigma^{-1} (f (\sigma (S), \sigma(m)))$, so we only need to compute $f$ for one pair $(S,m)$ in every equivalence class.
        \end{enumerate}

        We have implemented this approach in Mathematica \cite{wolfram_research_inc_mathematica_nodate}, the code is provided in \cite{hendrey_code_2023}. This code is fully commented, consult these for further implementation details. As a benchmark, this script can run in under 14 minutes on a 2020 MacBook Air with M1 chip and 16 GB ram running Mathematica 13.0.0.0.
        \renewcommand{\qedsymbol}{$\blacksquare$}
    \end{proof}

    In particular, there exists some $\mathcal G$-reachable state. This state is a certificate of the existence of two non-adjacent sets $C_1,C_2$ which induce connected graphs and both intersect $X$ and $Y$. From these, we may extract two non-adjacent $X\dash Y$-paths, as desired. This concludes the proof of the first part of the statement.

    We now show that this result is best possible, in two ways.

    First, consider the statement when we are given  four paths instead of five. By modifying the approach used above to prove the statement for five paths so as to it also taking into account paths which ``backtrack'', we were able to find a counter-example to this modified statement; it is shown in \cref{fig:4pathcounterexample}. A short Mathematica script which verifies that no pair of non-adjacent $X\dash Y$-paths exists in this graph is provided in \cite{hendrey_code_2023}.

    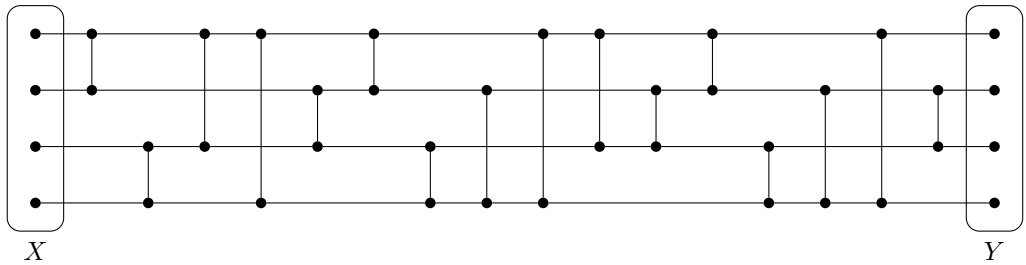
\begin{figure}
         \centering
         \begin{tikzpicture}[scale=0.75,
			dot/.style = {circle, fill, minimum size=#1,
			inner sep=0pt, outer sep=0pt},
			dot/.default = 4pt]

            \draw[rounded corners=5pt] (-0.5, 0.5) rectangle (0.5, 4.5) {};
            \node at (0,0.15) {$X$};

            \draw[rounded corners=5pt] (16.5, 0.5) rectangle (17.5, 4.5) {};
            \node at (17,0.15) {$Y$};
            
            \node[dot] (A0) at (0,4) {};
            \node[dot] (B0) at (0,3) {};
            \node[dot] (C0) at (0,2) {};
            \node[dot] (D0) at (0,1) {};

            \node[dot] (A1) at (1,4) {};
            \node[dot] (A2) at (3,4) {};
            \node[dot] (A3) at (4,4) {};
            \node[dot] (A4) at (6,4) {};
            \node[dot] (A5) at (9,4) {};
            \node[dot] (A6) at (10,4) {};
            \node[dot] (A7) at (12,4) {};
            \node[dot] (A8) at (15,4) {};

            \node[dot] (B1) at (1,3) {};
            \node[dot] (B2) at (5,3) {};
            \node[dot] (B3) at (6,3) {};
            \node[dot] (B4) at (8,3) {};
            \node[dot] (B5) at (11,3) {};
            \node[dot] (B6) at (12,3) {};
            \node[dot] (B7) at (14,3) {};
            \node[dot] (B8) at (16,3) {};

            \node[dot] (C1) at (2,2) {};
            \node[dot] (C2) at (3,2) {};
            \node[dot] (C3) at (5,2) {};
            \node[dot] (C4) at (7,2) {};
            \node[dot] (C5) at (10,2) {};
            \node[dot] (C6) at (11,2) {};
            \node[dot] (C7) at (13,2) {};
            \node[dot] (C8) at (16,2) {};

            \node[dot] (D1) at (2,1) {};
            \node[dot] (D2) at (4,1) {};
            \node[dot] (D3) at (7,1) {};
            \node[dot] (D4) at (8,1) {};
            \node[dot] (D5) at (9,1) {};
            \node[dot] (D6) at (13,1) {};
            \node[dot] (D7) at (14,1) {};
            \node[dot] (D8) at (15,1) {};

            \node[dot] (A9) at (17,4) {};
            \node[dot] (B9) at (17,3) {};
            \node[dot] (C9) at (17,2) {};
            \node[dot] (D9) at (17,1) {};

            \draw (A0) to (A1);
            \draw (A1) to (A2);
            \draw (A2) to (A3);
            \draw (A3) to (A4);
            \draw (A4) to (A5);
            \draw (A5) to (A6);
            \draw (A6) to (A7);
            \draw (A7) to (A8);
            \draw (A8) to (A9);

            \draw (B0) to (B1);
            \draw (B1) to (B2);
            \draw (B2) to (B3);
            \draw (B3) to (B4);
            \draw (B4) to (B5);
            \draw (B5) to (B6);
            \draw (B6) to (B7);
            \draw (B7) to (B8);
            \draw (B8) to (B9);

            \draw (C0) to (C1);
            \draw (C1) to (C2);
            \draw (C2) to (C3);
            \draw (C3) to (C4);
            \draw (C4) to (C5);
            \draw (C5) to (C6);
            \draw (C6) to (C7);
            \draw (C7) to (C8);
            \draw (C8) to (C9);

            \draw (D0) to (D1);
            \draw (D1) to (D2);
            \draw (D2) to (D3);
            \draw (D3) to (D4);
            \draw (D4) to (D5);
            \draw (D5) to (D6);
            \draw (D6) to (D7);
            \draw (D7) to (D8);
            \draw (D8) to (D9);

            \draw (A1) to (B1);
            \draw (C1) to (D1);
            \draw (A2) to (C2);
            \draw (A3) to (D2);
            \draw (B2) to (C3);
            \draw (A4) to (B3);
            \draw (C4) to (D3);
            \draw (B4) to (D4);
            \draw (A5) to (D5);
            \draw (A6) to (C5);
            \draw (B5) to (C6);
            \draw (A7) to (B6);
            \draw (C7) to (D6);
            \draw (B7) to (D7);
            \draw (A8) to (D8);
            \draw (B8) to (C8);

		\end{tikzpicture}
        \caption{Counter-example (a) in \cref{thm:twopathssubcubic}.}
        \label{fig:4pathcounterexample}
    \end{figure}

    Secondly, consider the statement if we replace the condition that every vertex in $V(\mathcal P)$ is incident to at most one edge in $E(G)\setminus E(\mathcal P)$ by the condition that the maximum degree is three. Of course, the difference only concerns vertices in $X$ and in $Y$, given that these vertices are the only ones with fewer than two neighbours in their path. The graph formed by a adding a matching between a five-cycle (with vertices $X$) and the complement of a copy of this five-cycle (with vertices $Y$), as shown in \cref{fig:5pathcounterexample}, is easily verified to not contain any pair of non-adjacent edges between $X$ and $Y$.\end{proof}

    \begin{figure}
         \centering
         \begin{tikzpicture}[scale=0.75,
			dot/.style = {circle, fill, minimum size=#1,
			inner sep=0pt, outer sep=0pt},
			dot/.default = 4pt]

            \draw[rounded corners=7pt] (0.25, 0.5) rectangle (1.5, 5.5) {};
            \node at (0.85,0.15) {$X$};

            \draw[rounded corners=7pt] (3.5, 0.5) rectangle (4.75, 5.5) {};
            \node at (4.15,0.15) {$Y$};
            
            \node[dot] (X1) at (1,1) {};
            \node[dot] (X2) at (1,2) {};
            \node[dot] (X3) at (1,3) {};
            \node[dot] (X4) at (1,4) {};
            \node[dot] (X5) at (1,5) {};

            \node[dot] (Y1) at (4,1) {};
            \node[dot] (Y2) at (4,2) {};
            \node[dot] (Y3) at (4,3) {};
            \node[dot] (Y4) at (4,4) {};
            \node[dot] (Y5) at (4,5) {};

            \draw (X1) to (Y1);
            \draw (X2) to (Y2);
            \draw (X3) to (Y3);
            \draw (X4) to (Y4);
            \draw (X5) to (Y5);

            \draw (X1) to (X2);
            \draw (X2) to (X3);
            \draw (X3) to (X4);
            \draw (X4) to (X5);
            \draw[bend right=23] (X5) to (X1);

            \draw[bend right] (Y1) to (Y3);
            \draw[bend right] (Y1) to (Y4);
            \draw[bend right] (Y2) to (Y4);
            \draw[bend right] (Y2) to (Y5);
            \draw[bend right] (Y3) to (Y5);

		\end{tikzpicture}
        \caption{Counter-example (b) in \cref{thm:twopathssubcubic}.}
        \label{fig:5pathcounterexample}
    \end{figure}
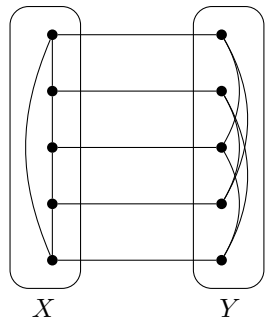

\bibliography{refs}
\bibliographystyle{abbrvurl}
\end{document}